\documentclass[11pt,reqno]{amsart}
\usepackage[margin=1in]{geometry}
\usepackage{amsmath,amssymb,amsthm,mathtools}
\usepackage[T1]{fontenc}
\usepackage{enumitem}
\usepackage{microtype}
\usepackage[colorlinks=true,linkcolor=blue,citecolor=blue,urlcolor=blue]{hyperref}

\newtheorem{theorem}{Theorem}[section]

\newtheorem{lemma}[theorem]{Lemma}
\newtheorem{remark}[theorem]{Remark}

\newcommand{\R}{\mathbb R}

\newcommand{\Sthree}{\mathcal S^3}
\newcommand{\K}{\mathcal K}
\newcommand{\Q}{\mathcal Q}
\newcommand{\tr}{\operatorname{tr}}
\newcommand{\supp}{\operatorname{supp}}
\newcommand{\BV}{BV}

\title[Unbounded variation solutions for uniformly elliptic equations]{Unbounded variation solutions for uniformly elliptic equations in nondivergence form in dimension three}

\author[N. Q. Le, Q. Sun, H. V. Tran]{Nam Q. Le, Qi Sun, Hung V. Tran}

\date{}

\thanks{
N. Q. Le is partially supported by NSF grant DMS-2452320.
H. V. Tran is partially supported by NSF grant DMS-2348305.
}

\address[N. Q. Le]
{Department of Mathematics, Indiana University, 831 E 3rd St, Bloomington, IN 47405}
\email{nqle@iu.edu}

\address[Q. Sun]
{Department of Mathematics, Rutgers University, 110 Frelinghuysen Road Piscataway, NJ 08854}
\email{qs176@math.rutgers.edu}

\address[H. V. Tran]
{Department of Mathematics, University of Wisconsin--Madison, Van Vleck Hall, 480 Lincoln Drive, Madison, Wisconsin 53706, USA}
\email{hung@math.wisc.edu}

\begin{document}
\begin{abstract}
For each nonnegative integer $m$, we construct smooth symmetric $3\times 3$ coefficient matrices $A_m$ satisfying the fixed ellipticity bound
\[
 I\leq A_m\leq 2^{81}I
\]
for which the smooth solutions of uniformly elliptic equations in nondivergence form
\[
 \tr(A_m(x)D^2 u_m)=A_m(x):D^2u_m=0\qquad\text{in }B_2\subset\R^3
\]
have common Dirichlet data, satisfy $\|u_m\|_{L^\infty(B_2)}\leq1$, but
\[
 \lim_{m\to \infty}\|Du_m\|_{L^1(B_1)}=\infty.
\]
Thus, there is no interior $W^{1,1}$ estimate depending only on ellipticity in dimension three, and consequently no such $W^{1,p}$ estimate for any $p\geq1$. 
This resolves in the negative an open question raised by Nadirashvili, Tkachev, and Vl\u{a}du\c{t}. 
The construction also gives a uniformly convergent limit $u\notin \BV_{\rm loc}(B_1)$ for a measurable uniformly elliptic coefficient matrix obtained as an $L^1$ limit of the $A_m$.
\end{abstract}

\maketitle

\section{Introduction}
In this paper, we are concerned with the validity of interior $W^{1,1}$ estimates for uniformly elliptic equations in nondivergence form in dimension three.

Let $B_r$ denote the ball of radius $r$ centered at the origin in $\R^3$, and let $\Sthree$ be the set of symmetric $3\times 3$ matrices.
Consider uniformly elliptic equations in nondivergence form
\begin{equation}\label{eq:main}
 \tr(A(x)D^2u)=A(x):D^2u=0\qquad\text{in }B_2\subset\R^3,
\end{equation}
where $A:B_2\to\Sthree$ and
\begin{equation}\label{eq:ell}
 I\leq A(x)\leq \Lambda I.
\end{equation}
The Krylov--Safonov theory gives a coefficient-independent interior H\"older estimate for bounded solutions of \eqref{eq:main}; see \cite{KrylovSafonov, SafonovKS} and \cite{GT}. 
Nadirashvili, Tkachev, and Vl\u{a}du\c{t} asked whether one also has
\begin{equation}\label{eq:question}
 \|Du\|_{L^1(B_1)}\leq C(\Lambda)\|u\|_{L^\infty(B_2)}
\end{equation}
for smooth coefficients, and whether solutions for measurable coefficients are locally in $W^{1,1}$; see \cite[Section 1.3 and Problem 1.3.1]{NTV}. 
To the best of our knowledge, the three-dimensional question \eqref{eq:question} has remained open. 
In two dimensions, much stronger regularity is available; see, for example, \cite{Bernstein,BaernsteinKovalev,NTV}.

\medskip
There is a positive near-isotropic regime. 
In dimension three, \eqref{eq:ell} implies the Cordes condition whenever $\Lambda<4$, since
\[
 \frac{|A|^2}{(\tr A)^2}
 \leq \frac{\Lambda^2+2}{(\Lambda+2)^2}<\frac12.
\]
The Cordes--Campanato estimate then gives an interior $W^{2,2}$ bound; see \cite{Cordes,Campanato}. The point here is that, for one fixed finite ellipticity ratio, \eqref{eq:question} fails completely.
Here is our main result, which resolves in the negative the above open question of Nadirashvili, Tkachev, and Vl\u{a}du\c{t}. 

\begin{theorem}[Failure of interior $W^{1,1}$ estimates]
\label{thm:smooth}
There exist sequences $u_m\in C^\infty(B_2)$, $A_m\in C^\infty(B_2;\Sthree)$ such that
\begin{align*}
 &I\leq A_m\leq 2^{81}I,\\
 &A_m:D^2u_m=0\qquad\text{in }B_2,    
\end{align*}
all $u_m$ have the same Dirichlet data $1-x_1^2/2$ on $\partial B_2$, and
\[
\begin{cases}
\|u_m\|_{L^\infty(B_2)}\leq1,\\
\lim_{m\to \infty}\|Du_m\|_{L^1(B_1)}=\infty.
\end{cases}
\]
Moreover, all $u_m$ agree with one fixed quadratic polynomial outside the cube
\[
 Q_0=\left(-\frac12,\frac12\right)^3\Subset B_1.
\]
\end{theorem}

Thus, while uniform ellipticity gives universal interior H\"older regularity by the Krylov--Safonov theory, in dimension three it gives not even weak-$L^1$ control of the gradient; see Remark~\ref{rem:weakL1}.

\medskip
The construction in Theorem \ref{thm:smooth} is simple at the level of scales. 
Starting from
\[
 M_0=\operatorname{diag}(-1,1,1),
\]
three localized rank-one splittings produce, on a fixed fraction of a cube, the exact Hessian $-cM_0$, where $c>1$ is an appropriate fixed amplification parameter. 
One cycle therefore multiplies the Hessian by $c$ and reduces the spatial scale by $c^{-1/2}$.
Thus,
\[
 \text{size of }u:\quad c(c^{-1/2})^2\sim1,
 \qquad
 \text{size of }Du:\quad c(c^{-1/2})\sim\sqrt c.
\]
An additional fixed contraction makes the changes of $u$ summable, while the $L^1$ mass of the gradient grows geometrically. 
The key PDE point is that, throughout the localization, the Hessian keeps one quantitatively positive and one quantitatively negative eigenvalue. 
This allows every Hessian produced by the construction to be annihilated by a coefficient matrix in one fixed ellipticity class.

\medskip
The geometric part of the argument is laminate-like and is naturally related to convex integration. 
Rank-one splittings and their iteration are basic in convex integration; see, for example, \cite{MullerSverak99,MullerSverak03, AFSz}. 
The use of rank-one decompositions to produce counterexamples to $L^1$ estimates is especially close in spirit to Conti--Faraco--Maggi \cite{ContiFaracoMaggi}. 
Our construction is not a direct application of these results: the oscillating object is a scalar Hessian, the available rank-one directions are the second-order directions $e_i\otimes e_i$, and the localization must preserve a quantitative saddle condition to retain uniform ellipticity.

\medskip
We also record the measurable-coefficient limit. 
Approximation solutions for uniformly elliptic equations with measurable coefficients have a subtle theory in dimensions at least three; in particular, Nadirashvili \cite{Nadirashvili97} and Safonov \cite{Safonov99} constructed nonuniqueness phenomena for limits of smooth uniformly elliptic problems. 
Earlier differentiability questions for such equations were studied by Nadirashvili \cite{Nadirashvili86}; see also \cite{Nadirashvili11} for the relation between derivatives of fully nonlinear viscosity solutions and linear equations with measurable coefficients. 
We say here that $u$ is an $L^1$-approximation solution of \[A(x):D^2u=0\]
if there are smooth coefficient matrices $A_m$ with a common ellipticity bound and classical solutions $u_m$ of \[A_m(x):D^2u_m=0\] such that $u_m\to u$ uniformly and $A_m\to A$ in $L^1$. 
This is the approximation notion used in \cite{Nadirashvili11,NTV}; see also Jensen \cite{Jensen}. 
It is different from the $L^p$-viscosity framework for equations with measurable ingredients developed in \cite{CCKS}.

\begin{theorem}[Unbounded variation solutions]
\label{thm:limit}
The sequences in Theorem \ref{thm:smooth} may be chosen so that
\[
 u_m\longrightarrow u\quad\text{uniformly in }B_2,
 \qquad
 A_m\longrightarrow A\quad\text{in }L^1(B_2),
\]
where
\[
 I\leq A\leq2^{81}I\qquad\text{a.e. in }B_2
\]
and
\[
 u\notin \BV_{\rm loc}(B_1).
\]
In particular, $u$ is an $L^1$-approximation solution of \[A(x):D^2u=0\qquad\text{in }B_2.\]
\end{theorem}

Since an $L^p$ bound on $Du$ with $p\geq1$ implies an $L^1$ bound on $B_1$, Theorem \ref{thm:smooth} rules out every ellipticity-only interior $W^{1,p}$ estimate, $p\geq1$, and every corresponding $BV$ estimate. By adding dummy variables and using block-diagonal coefficients, the failure of an interior $W^{1,1}$ estimate immediately extends to every dimension $n\geq3$.

\subsection*{Notations}
We use $|M|$ to denote the Hilbert--Schmidt norm of a matrix $M$ and $A:B=\tr(AB)$ for $A, B\in \Sthree$. All boxes below are axis-parallel. The numerical constants are chosen to keep the construction explicit and reasonably short. No claim is made that the final ellipticity ratio $2^{81}$ is optimal; the choices below are near-optimized within this particular three-split implementation while keeping simple rational parameters.

\subsection*{Organization of the paper}
The paper is organized as follows.
In Section \ref{sec:local split}, we give some preliminaries and the local splitting lemma (Lemma \ref{lem:split}).
Section \ref{sec:3-step} is devoted to the main mechanism of the construction, the three-step amplification (Lemma \ref{lem:cycle}).
The proof of Theorem \ref{thm:smooth} is given in Section \ref{sec:thm1}.
We give the proof of Theorem \ref{thm:limit} in Section \ref{sec:thm2}.

\subsection*{AI assistance} The main results of this paper were obtained through a series of chats with ChatGPT 5.6 Sol. The key strategies were obtained by ChatGPT. 
The authors reworked and rewrote the article entirely. 
All arguments have been checked and simplified by the authors.
We take full responsibility for its correctness and content.

\section{Balanced Hessians and a local split}\label{sec:local split}

For $M\in\Sthree$, let $\lambda_1(M), \lambda_2(M), \lambda_3(M)$ be its eigenvalues and 
\[
 \lambda_+(M)=\sum_{\lambda_i(M)>0}\lambda_i(M),
 \qquad
 \lambda_-(M)=\sum_{\lambda_i(M)<0}|\lambda_i(M)|.
\]
For $\Gamma\geq1$, set
\[
 \K_\Gamma
 =\left\{M\in\Sthree:\ \lambda_+(M),\lambda_-(M)>0,
 \quad \Gamma^{-1}\leq\frac{\lambda_+(M)}{\lambda_-(M)}\leq\Gamma\right\}.
\]
Thus, $\K_\Gamma$ consists of quantitatively balanced indefinite matrices.

\begin{lemma}\label{lem:criterion}
For non-zero $M\in\Sthree$ and $\Gamma\geq 1$, the following are equivalent:
\begin{enumerate}[label=(\roman*)]
\item $M\in\K_\Gamma$;
\item there exists $A\in\Sthree$ such that
\[
 I\leq A\leq\Gamma I,
 \qquad A:M=0.
\]
\end{enumerate}
\end{lemma}
\begin{proof}
Without loss of generality, after an orthogonal change of basis, we may assume
\[
 M=\operatorname{diag}(\lambda_1,\lambda_2,\lambda_3).
\]

Suppose first that there exists \(A\in\Sthree\) such that
\[
 I\leq A\leq\Gamma I,
 \qquad A:M=0.
\]
Writing \(A=(a_{ij})\) in this basis, we have $1\leq a_{ii}\leq\Gamma$ for every \(i\). Since \(M\) is diagonal,
\[
 A:M=\sum_{i=1}^3a_{ii}\lambda_i,
\]
and hence
\[
 \lambda_+(M)-\Gamma\lambda_-(M)
 \leq A:M
 \leq\Gamma\lambda_+(M)-\lambda_-(M).
\]
Using \(A:M=0\), we obtain
\[
 \lambda_+(M)\leq\Gamma\lambda_-(M),
 \qquad
 \lambda_-(M)\leq\Gamma\lambda_+(M).
\]
Since \(M\neq0\), these inequalities imply
\(\lambda_+(M),\lambda_-(M)>0\). 
Therefore
\[
 \Gamma^{-1}
 \leq\frac{\lambda_+(M)}{\lambda_-(M)}
 \leq\Gamma,
\]
so \(M\in\K_\Gamma\).

Conversely, suppose that \(M\in\K_\Gamma\). If
\(\lambda_+(M)\geq\lambda_-(M)\), define
\[
 A=\operatorname{diag}(a_1,a_2,a_3),
 \qquad
 a_i=
 \begin{cases}
  1, \qquad& \lambda_i\geq0,\\[2mm]
  \dfrac{\lambda_+(M)}{\lambda_-(M)}, \qquad& \lambda_i<0.
 \end{cases}
\]
Then \(I\leq A\leq\Gamma I\), and
\[
 \begin{aligned}
 A:M
 &=
 \sum_{\lambda_i>0}\lambda_i
 +\frac{\lambda_+(M)}{\lambda_-(M)}
  \sum_{\lambda_i<0}\lambda_i\\
 &=\lambda_+(M)
 -\frac{\lambda_+(M)}{\lambda_-(M)}\lambda_-(M)
 =0.
 \end{aligned}
\]
If \(\lambda_+(M)<\lambda_-(M)\), define instead
\[
 a_i=
 \begin{cases}
  \dfrac{\lambda_-(M)}{\lambda_+(M)}, \qquad& \lambda_i>0,\\[2mm]
  1, \qquad& \lambda_i\leq0.
 \end{cases}
\]
Again \(I\leq A\leq\Gamma I\), and
\[
 A:M
 =
 \frac{\lambda_-(M)}{\lambda_+(M)}
 \sum_{\lambda_i>0}\lambda_i
 +\sum_{\lambda_i<0}\lambda_i
 =0.
\]
Thus the two conditions are equivalent.
\end{proof}
We need the coefficient field to depend smoothly on the Hessian. 
The following explicit version of the smooth selection lemma will be useful.
\begin{lemma}\label{lem:selection}
For every $1\leq \Gamma<\infty$, there are an open conic neighborhood $\mathcal U_\Gamma$ of $\K_\Gamma$ and a smooth, even, degree-zero map
\[
 \mathcal A:\mathcal U_\Gamma\longrightarrow\Sthree
\]
such that, for every $M\in\K_\Gamma$,
\begin{equation}\label{eq:selector}
 I\leq\mathcal A(M)\leq(\Gamma+1)I,
 \qquad
 \mathcal A(M):M=0.
\end{equation}
\end{lemma}

\begin{proof}
It is enough to work on the compact set
\[
 \Sigma_\Gamma=\{M\in\K_\Gamma:|M|=1\}.
\]
For $M_*\in\Sigma_\Gamma$, choose $A_*\in \Sthree$ from Lemma \ref{lem:criterion}. Put
\[
 \tau=\frac1{\Gamma+1},
 \qquad
 B_*=(1+\tau)A_*,
\]
and, for $M\in\Sthree, M\neq0$, define
\[
 B_{M_*}(M)=B_*-\frac{B_*:M}{|M|^2}M.
\]
Then, $B_{M_*}(M):M=0$, and the map is even and homogeneous of degree zero. 
Since $B_*:M_*=0$, after shrinking to a conic neighborhood of the antipodal pair $\{M_*,-M_*\}$ we may assume that the correction term has operator norm less than $\tau$. 
Hence, in this neighborhood,
\[
 I\leq B_{M_*}(M)\leq(\Gamma+1)I.
\]
A finite antipodally symmetric cover of $\Sigma_\Gamma$ and a smooth even partition of unity give $\mathcal A$ by convex combination. 
Then, extending along rays gives the desired conic neighborhood.
\end{proof}

\subsection*{Cell and cutoff} 
We next fix the one-dimensional cell and the transverse cutoff. 

Let $\rho\in C_c^\infty((-1,1))$ be nonnegative, even, and satisfy $\int_{-1}^1\rho(t)\,dt=1$.
For $h>0$, denote by $\rho_h(t)=h^{-1}\rho(t/h)$. 
Define the step function $g_0$, which is even about $1/2$,
\[
 g_0(t)=
 \begin{cases}
 1,\qquad&\displaystyle \frac1{256}<t<\frac{129}{512}
       \ \text{or}\ \frac{383}{512}<t<\frac{255}{256},\\[2mm]
 -1,\qquad&\displaystyle \frac{129}{512}<t<\frac{383}{512},\\[2mm]
 0,\qquad&\text{otherwise},
 \end{cases}
\]
extended by zero to $\R$, and set
\[
 g=\rho_{1/1024}*g_0.
\]
Then, $g\in C_c^\infty((0,1))$, $|g|\leq1$, and symmetry together with the equality of the positive and negative masses gives
\begin{equation}\label{eq:moments}
 \int_0^1g(t)\,dt=\int_0^1t g(t)\,dt=0.
\end{equation}
Moreover, $g=-1$ on a central interval about $1/2$ of length $253/512$. We retain a slightly smaller interval and fix
\begin{equation}\label{eq:beta}
 \beta=\frac{63}{128}.
\end{equation}
Define
\[
 F(t)=\int_0^t(t-s)g(s)\,ds.
\]
By \eqref{eq:moments}, $F\in C_c^\infty((0,1))$, and
\[
 \|F''\|_{L^\infty}\leq1,
 \qquad
 \|F'\|_{L^\infty}\leq1,
 \qquad
 \|F\|_{L^\infty}\leq1.
\]
Thus, for an interval \(I=(a,a+\ell)\) of length \(\ell\) and
\(d>0\), define
\[
 f(x)=d\ell^2F\left(\frac{x-a}{\ell}\right).
\]
The resulting cell $f$ satisfies
\begin{equation}\label{eq:cellbounds}
 |f''|\leq d,
 \qquad
 |f'|\leq d\ell,
 \qquad
 |f|\leq d\ell^2,
\end{equation}
and $f''=-d$ on a subinterval of length $\beta\ell$. 
Replacing $f$ by $-f$ gives the opposite sign. 
The cells are flat near their endpoints and may therefore be concatenated smoothly.

\medskip
For the transverse cutoff, let
\[
 h(s)=1-3s^2+2s^3\qquad(0\leq s\leq1),
\]
and
\[
 a=\frac5{32}+\frac{1}{1024},
 \qquad
 b=\frac{15}{32}-\frac{1}{1024}.
\]
Define an even $C^1$ function $\chi_0$ by $\chi_0=1$ on $[-a,a]$, $\chi_0=0$ outside $(-b,b)$, and on $a<|t|<b$ set
\[
 \chi_0(t)=h\left(\frac{|t|-a}{b-a}\right).
\]
Let $\chi=\rho_{1/1024}*\chi_0$.
Then, $\chi\in C_c^\infty((-1/2,1/2))$,
\begin{equation}\label{eq:cutoff}
 0\leq\chi\leq1,
 \qquad
 \chi=1\quad\text{on }\left[-\frac5{32},\frac5{32}\right],
 \qquad
 \supp\chi\subset\left(-\frac{15}{32},\frac{15}{32}\right),
\end{equation}
and, since $b-a=159/512$,
\begin{equation}\label{eq:cutoffder}
 \|\chi'\|_{L^\infty}\leq5,
 \qquad
 \|\chi''\|_{L^\infty}\leq63.
\end{equation}
Inside the region where $\chi=1$, we retain the fixed central fraction
\begin{equation}\label{eq:gamma}
\gamma=\frac3{10}.
\end{equation}

\begin{lemma}[Local splitting]\label{lem:split}
Let
\[
 Q=I_i\times I_j\times I_k,
 \qquad
 M=\operatorname{diag}(m_i,m_j,m_k),
 \qquad m_jm_k<0.
\]
Let $R_i,R_j,R_k$ be the three side lengths, with $R_i$ corresponding to the splitting direction. 
Assume, for some $S,a,d, L>0$, that
\begin{equation}\label{eq:splitass}
 \min(|m_j|,|m_k|)\geq aS,
 \qquad
 |M|+d\leq LS.
\end{equation}
Set
\begin{equation}\label{eq:eps0}
\varepsilon_0=\frac5{378}.
\end{equation}
If
\begin{equation}\label{eq:scalesplit}
d\ell^2\leq\min(\varepsilon_0|m_j|R_j^2,\varepsilon_0|m_k|R_k^2),
 \qquad
 16\ell \leq R_i,
\end{equation}
then, for either choice of sign, there exists $w\in C_c^\infty(Q)$ such that
\begin{enumerate}[label=(\roman*)]
\item
\begin{equation}\label{eq:localsplitcone}
 M+D^2w\in\K_{16\sqrt2\,L/a}\qquad\text{in }Q;
\end{equation}
\item on a union of rectangular boxes of total volume at least
\begin{equation}\label{eq:theta}
\theta|Q|,\qquad
 \theta=\frac{15}{16}\,\beta\gamma^2
       =\frac{1701}{40960},
\end{equation}
one has
\[
 M+D^2w=M\pm d\,e_i\otimes e_i;
\]
\item
\[
 \|w\|_{L^\infty(Q)}\leq d\ell^2.
\]
\end{enumerate}
The exact boxes have fixed positive margins inside regions on which the Hessian is exactly constant. Their two transverse side lengths are $\gamma R_j$ and $\gamma R_k$, and their normal side length is $\beta\ell$.
\end{lemma}

\begin{proof}
Let $x_j^0,x_k^0$ be the centers of the two transverse intervals and define
\[
 \eta(x_j,x_k)
 =\chi\left(\frac{x_j-x_j^0}{R_j}\right)
  \chi\left(\frac{x_k-x_k^0}{R_k}\right).
\]
Concatenate the one-dimensional cells of length $\ell$ in the $i$-direction and call the resulting function $f=f(x_i)$. Put $w=\eta f$. In the decomposition $\R e_i\oplus e_i^\perp$,
\begin{equation}\label{eq:block}
 M+D^2w=
 \begin{pmatrix}
 m_i+\eta f'' & f'(D\eta)^T\\
 f'D\eta & \operatorname{diag}(m_j,m_k)+fD^2\eta
 \end{pmatrix}.
\end{equation}
By \eqref{eq:cutoffder},
\[
 |D_{jj}\eta|\leq\frac{63}{R_j^2},\qquad
 |D_{kk}\eta|\leq\frac{63}{R_k^2},\qquad
 |D_{jk}\eta|\leq\frac{25}{R_jR_k}.
\]
Since $63\varepsilon_0=5/6$, \eqref{eq:cellbounds} and \eqref{eq:scalesplit} show that the two diagonal entries of the lower-right block in \eqref{eq:block} keep their opposite signs and have absolute values at least $|m_j|/6$ and $|m_k|/6$. 
Hence, this $2\times2$ block is indefinite. 
The Cauchy interlacing theorem gives
\begin{equation}\label{eq:eigenlower}
 \lambda_{\max}(M+D^2w)\geq\frac{aS}{6},
 \qquad
 \lambda_{\min}(M+D^2w)\leq-\frac{aS}{6}.
\end{equation}
We next give the upper bound explicitly. 
Let $v=f'D\eta$ and $q=fD^2\eta$. 
Since $25\varepsilon_0=125/378<1/3$, the scale assumptions give
\[
 |v|^2
 \leq25\varepsilon_0d(|m_j|+|m_k|)
 <\frac{\sqrt2}{3}L^2S^2.
\]
Moreover,
\[
 |q|^2
 \leq\left(\frac{25}{36}+\frac19\right)(m_j^2+m_k^2)
 =\frac{29}{36}(m_j^2+m_k^2)
 \leq\frac{29}{36}L^2S^2.
\]
Therefore,
\[
 |D^2w|^2
 \leq\left(1+\frac{2\sqrt2}{3}+\frac{29}{36}\right)L^2S^2
 <\frac{25}{9}L^2S^2,
\]
and hence
\begin{equation}\label{eq:Hupper}
 |M+D^2w|\leq\frac83LS.
\end{equation}

Since $M+D^2w$ is indefinite in dimension three, it has at most two positive and at most two negative eigenvalues. Hence, by the Cauchy--Schwarz inequality,
\[
\lambda_+(M+D^2w)\le \sqrt{2}\,|M+D^2w|,
\qquad
\lambda_-(M+D^2w)\le \sqrt{2}\,|M+D^2w|.
\]

It follows from \eqref{eq:eigenlower} that
\[
\lambda_+(M+D^2w),\lambda_-(M+D^2w)
\geq \frac{aS}{6}.
\]
On the other hand, by \eqref{eq:Hupper} and the preceding
Cauchy--Schwarz estimates,
\[
\lambda_+(M+D^2w),\lambda_-(M+D^2w)
\leq \sqrt{2}\,|M+D^2w|
\leq \frac{8\sqrt{2}}{3}LS.
\]
Therefore,
\[
\frac{\lambda_+(M+D^2w)}
     {\lambda_-(M+D^2w)}
\leq
\frac{(8\sqrt{2}/3)LS}{aS/6}
=
\frac{16\sqrt{2}\,L}{a},
\]
and similarly,
\[
\frac{\lambda_-(M+D^2w)}
     {\lambda_+(M+D^2w)}
\leq
\frac{16\sqrt{2}\,L}{a}.
\]
Hence
\[
M+D^2w\in \mathcal K_{16\sqrt2\,L/a},
\]
which proves~(i).

On the region where $\eta=1$ and $f''=\pm d$, the Hessian is exactly $M\pm d e_i\otimes e_i$. 
Since $R_i\geq16\ell$, complete cells fill at least $15/16$ of the normal interval. 
In every complete cell we retain the central plateau of relative length $\beta$, and in each transverse direction we retain the central fraction $\gamma$. 
This gives exactly the volume fraction in \eqref{eq:theta}. 
Finally, $|w|\leq|f|\leq d\ell^2$. The lemma is proved.
\end{proof}

\section{The three-step amplification}\label{sec:3-step}

Fix $c>1$ and set
\[
 M_0=\operatorname{diag}(-1,1,1).
\]
Introduce
\[
 \begin{aligned}
 M_1&=\operatorname{diag}(-1,1,-c),\qquad
 &\widetilde M_1&=\operatorname{diag}(-1,1,c+2),\\
 M_2&=\operatorname{diag}(c,1,-c),\qquad
 &\widetilde M_2&=\operatorname{diag}(-c-2,1,-c),\\
 M_3&=\operatorname{diag}(c,-c,-c),\qquad
 &\widetilde M_3&=\operatorname{diag}(c,c+2,-c).
 \end{aligned}
\]
Then,
\begin{equation}\label{eq:threeav}
 M_0=\frac{M_1+\widetilde M_1}{2},
 \qquad
 M_1=\frac{M_2+\widetilde M_2}{2},
 \qquad
 M_2=\frac{M_3+\widetilde M_3}{2},
\end{equation}
and, crucially,
\begin{equation}\label{eq:amp}
 M_3=-cM_0.
\end{equation}
The three active increments are
\begin{equation}\label{eq:increments}
 -(c+1)e_3\otimes e_3,
 \qquad
 (c+1)e_1\otimes e_1,
 \qquad
 -(c+1)e_2\otimes e_2.
\end{equation}
Thus, all three splittings use the same symmetric one-dimensional cell. 
In the splitting direction, the two transverse entries of the parent matrix always have opposite signs:
\[
 (-1,1),\qquad (1,-c),\qquad (c,-c).
\]
This is the reason localization preserves uniform ellipticity.

We now fix the geometric parameter
\begin{equation}\label{eq:delta}
\delta=\frac1{25}.
\end{equation}
For a cube of side $r$, we use the three cell lengths
\begin{equation}\label{eq:ell123}
 \ell_1=\frac{\delta r}{\sqrt c},
 \qquad
 \ell_2=\frac{\delta^2r}{\sqrt c},
 \qquad
 \ell_3=\frac{\delta^3r}{\sqrt c}.
\end{equation}
We also set
\begin{equation}\label{eq:kappaphi}
\kappa=\beta\delta^3=\frac{63}{2,000,000},
 \qquad
 \phi=\theta^3\frac{\delta^3}{\gamma^3}
 =\frac{3^{12}7^3}{2^{36}5^6}.
\end{equation}
Finally put
\begin{equation}\label{eq:sigma}
\sigma=\frac1{64}.
\end{equation}

\begin{lemma}[One amplification cycle]\label{lem:cycle}
For every cube $Q$ of side $r$, every $s\neq0$, and every $c\geq1$, there exists $V_{Q,s}\in C_c^\infty(Q)$ with the following properties. 
Set
\[
 \rho=\frac{\kappa}{\sqrt c}
\]
and
\begin{equation}\label{eq:Gammac}
 \Gamma(c)=16\sqrt2\bigl(\sqrt{c^2+2}+c+1\bigr).
\end{equation}
Then:
\begin{enumerate}[label=(\roman*)]
\item
\[
 sM_0+D^2V_{Q,s}\in\K_{\Gamma(c)}\qquad\text{in }Q;
\]
\item there are pairwise disjoint cubes $Q_\nu\Subset Q$, all of side $\rho r$, such that
\[
 \sum_\nu|Q_\nu|\geq\phi|Q|
\]
and, on a neighborhood of every $Q_\nu$,
\[
 sM_0+D^2V_{Q,s}=-csM_0;
\]
\item there is a cube $G_Q\Subset Q$ of side $\sigma r$ on a neighborhood of which $V_{Q,s}=0$. 
All child cubes $Q_\nu$ lie in the central $x_1$-slab of width $\gamma r=3r/10$, which is disjoint from $G_Q$;
\item
\begin{equation}\label{eq:Vbound}
 \|V_{Q,s}\|_{L^\infty(Q)}\leq\frac1{300}|s|r^2.
\end{equation}
\end{enumerate}
\end{lemma}

\begin{proof}
Let $S=|s|$ and $d=(c+1)S$. 
Apply Lemma \ref{lem:split} successively in the directions $e_3,e_1,e_2$, following the active matrices
\[
 sM_0\longrightarrow sM_1\longrightarrow sM_2\longrightarrow sM_3=-csM_0.
\]
Think of this three-step amplification as one rotation cycle. After each cycle, the Hessian is amplified by a multiplicative factor of $-c$ in every $Q_\nu$.

We first verify the scale conditions. 
Since $d=(c+1)S\leq2cS$, the worst inequality among the six transverse inequalities in \eqref{eq:scalesplit} is
\[
 \frac{2\delta^2}{\beta^2}
 =\frac{32768}{2480625}
 <\frac5{378}=\varepsilon_0.
\]
Indeed, this is the short transverse direction at the second and third stages; all other scale inequalities are smaller. The other condition $R_i\geq16\ell_i$ follows directly from $\delta=1/25$, $\gamma=3/10$, and $c\geq1$. 
Thus, Lemma \ref{lem:split} applies at all three stages.

The exact boxes after the first stage have side lengths
\[
 (\gamma r,\gamma r,\beta\ell_1)
\]
in the $(x_1,x_2,x_3)$ coordinates. 
After the second stage they have side lengths
\[
 (\beta\ell_2,\gamma^2r,\gamma\beta\ell_1),
\]
and after the third stage they have side lengths
\begin{equation}\label{eq:finalrect}
 (\gamma\beta\ell_2,\beta\ell_3,\gamma^2\beta\ell_1).
\end{equation}
Since $\delta<\gamma$, the shortest side in \eqref{eq:finalrect} is $\beta\ell_3$. 
Each final rectangle therefore contains a cube of side
\[
 \beta\ell_3=\frac{\kappa r}{\sqrt c}.
\]
The three local splittings leave total volume at least $\theta^3|Q|$ in the final exact rectangles. 
Passing from each rectangle in \eqref{eq:finalrect} to the selected cube costs the factor $\delta^3/\gamma^3$. 
Hence, the child cubes have total volume at least $\phi|Q|$, proving (ii).

For the cone bound, the three values of $L/a$ in Lemma \ref{lem:split} can be taken as
\[
 \sqrt3+c+1,
 \qquad
 \sqrt{c^2+2}+c+1,
 \qquad
 \frac{\sqrt{2c^2+1}+c+1}{c}.
\]
(The corresponding values of $a$ can be taken as $1, 1, c$.) The second is the largest for $c\geq1$, and (i) follows from \eqref{eq:localsplitcone} and \eqref{eq:Gammac}. 

The first transverse cutoff is supported in the central $15/16$ of the $x_1$-interval, leaving an $x_1$ boundary collar of width $r/32$. 
All later perturbations and all child cubes lie inside the first exact core, hence in the central $x_1$-slab of width $\gamma r$. 
We may choose a cube of side $r/64$ inside the left boundary collar, which proves (iii).

Finally, supports at each of the three levels are pairwise disjoint. 
By Lemma \ref{lem:split}(iii),
\[
 \|V_{Q,s}\|_{L^\infty(Q)}
 \leq 2Sr^2(\delta^2+\delta^4+\delta^6)
 =\frac{782502}{244140625}Sr^2
 <\frac1{300}Sr^2,
\]
which proves (iv).
\end{proof}

\textbf{Choice of $c$.} We now make the final amplification parameter explicit:
\begin{equation}\label{eq:cchoice}
c=2^{75}.
\end{equation}
The reason for this choice is that
\begin{equation}\label{eq:phikappa}
 \phi\kappa
 =\frac{3^{14}7^4}{2^{43}5^{12}},
\end{equation}
and
\[
 c(\phi\kappa)^2
 =\frac1{2^{11}}
 \left(\frac{3^7 7^2}{5^6}\right)^4
 >\frac{27^4}{2^{19}}>1.
\]
Here we used $3^7 7^2/5^6=107163/15625>27/4$. Consequently,
\begin{equation}\label{eq:growthfactor}
\phi\kappa\sqrt c>1.
\end{equation}

\section{Iteration and gradient blow-up}\label{sec:thm1}
We now iterate the amplification cycle from Lemma~\ref{lem:cycle},
prove the resulting gradient blow-up, and construct the corresponding
smooth uniformly elliptic coefficient matrices.
\begin{proof}[Proof of Theorem \ref{thm:smooth}]
Fix
\[
 Q_0=\left(-\frac12,\frac12\right)^3,
 \qquad |Q_0|=1,
\]
so $Q_0\Subset B_1$, and set
\begin{equation}\label{eq:q}
 q(x)=\frac14x\cdot M_0x,
 \qquad
 D^2q=\frac12M_0.
\end{equation}
Put
\[
 s_m=\frac12(-c)^m,
 \qquad
 \Q_0=\{Q_0\},
 \qquad
 u_0=q.
\]
Inductively, for $m\geq 0$, after $\Q_m$ and $u_m$ are defined, set
\begin{equation}\label{eq:iteration}
 u_{m+1}=u_m+\sum_{Q\in\Q_m}V_{Q,s_m},
\end{equation}
and let $\Q_{m+1}$ be the family of child cubes given by Lemma \ref{lem:cycle}. Write
\[
 \Omega_m=\bigcup_{Q\in\Q_m}Q.
\]
Every cube in $\Q_m$ has side
\begin{equation}\label{eq:rm}
 r_m=\rho^m,
 \qquad
 \rho=\frac{\kappa}{\sqrt c},
\end{equation}
and
\begin{equation}\label{eq:volumelower}
 |\Omega_m|\geq\phi^m.
\end{equation}
Moreover, on a neighborhood of every $Q\in\Q_m$,
\begin{equation}\label{eq:exactm}
 D^2u_m=s_mM_0=\frac{1}{2}(-c)^m M_0.
\end{equation}
Since all perturbations are compactly supported in $Q_0$, every $u_m$ is smooth in $\R^3$ and agrees with $q$ outside $Q_0$.

By \eqref{eq:Vbound},
\begin{equation}\label{eq:uniformincrement}
 \|u_{m+1}-u_m\|_{L^\infty}
 \leq\frac1{300}|s_m|r_m^2
 =\frac1{600}\kappa^{2m}.
\end{equation}
Hence, $(u_m)$ is a Cauchy sequence in $C(B_2)$ and it converges uniformly to some $u\in C(B_2)$.

\medskip

We now estimate the gradients. Let $Q\in\Q_m$ and let $x_Q$ be its center. From \eqref{eq:exactm},
\[
 Du_m(x_Q+y)=s_mM_0y+b_Q
\]
for some $b_Q\in\R^3$. Since $Q$ is symmetric and $M_0$ is orthogonal,
\begin{align*}
 \int_Q|Du_m|\,dx
 &=\frac12\int_Q\bigl(|s_mM_0y+b_Q|+|-s_mM_0y+b_Q|\bigr)\,dy\\
 &\geq |s_m|\int_Q|y|\,dy.
\end{align*}
For a cube of side $r_m$,
\begin{equation}\label{eq:ycube}
 \int_Q|y|\,dy\geq\frac18r_m|Q|,
\end{equation}
because the set $\{|y_1|\geq r_m/4\}$ has measure $|Q|/2$. 
Thus,
\[
 \int_Q|Du_m|\,dx
 \geq\frac1{16}c^m r_m|Q|.
\]
Summing over $Q\in\Q_m$ and using \eqref{eq:volumelower} gives the explicit bound
\begin{equation}\label{eq:gradblow}
 \ \int_{B_1}|Du_m|\,dx
 \geq\frac1{16}(\phi\kappa\sqrt c)^m
 \longrightarrow\infty.
\end{equation}

It remains to produce smooth uniformly elliptic coefficients. Set
\begin{equation}\label{eq:Gammafinal}
 \Gamma=\Gamma(c)
 =16\sqrt2\bigl(\sqrt{c^2+2}+c+1\bigr).
\end{equation}
We claim inductively that, for all $m\geq 0$ and $x\in \R^3$,
\begin{equation}\label{eq:Hessiancone}
 D^2u_m(x)\in\K_\Gamma.
\end{equation}
Indeed, $D^2u_0=M_0/2\in\K_2\subset\K_\Gamma$. At the $m$-th step, $D^2u_m=s_mM_0$ on a neighborhood of every active cube $Q\in\Q_m$. Hence, inside $Q$,
\[
 D^2u_{m+1}=s_mM_0+D^2V_{Q,s_m}\in\K_\Gamma
\]
by Lemma \ref{lem:cycle}, while outside the active cubes $D^2u_{m+1}=D^2u_m$. Since the active cubes are disjoint and the perturbations are compactly supported in them, \eqref{eq:Hessiancone} follows.

Define
\begin{equation}\label{eq:Am}
 A_m(x)=\mathcal A(D^2u_m(x)),
\end{equation}
where $\mathcal A$ is the map from Lemma \ref{lem:selection}. 
Then
\begin{align*}
 & I\leq A_m\leq(\Gamma+1)I,\\
&A_m:D^2u_m=0.   
\end{align*}
The Hessians never vanish, so $A_m$ is smooth. 
For the fixed choice $c=2^{75}$,
\[
 \Gamma+1
 <48(c+1)+1
 <64c
 =2^{81},
\]
where we used $\sqrt{c^2+2}<c+1$ and $\sqrt2<3/2$.

\medskip

All $u_m$ have boundary value $q|_{\partial B_2}$. 
On $\partial B_2$,
\[
 q(x)=1-\frac{x_1^2}{2},
\]
so $|q|\leq1$. 
The maximum principle therefore gives
\[
 \|u_m\|_{L^\infty(B_2)}\leq1.
\]
Together with \eqref{eq:gradblow}, this proves Theorem \ref{thm:smooth}.
\end{proof}

\section{The measurable-coefficient limit}\label{sec:thm2}
We now pass to the limit in the preceding construction and use the
untouched cubes from Lemma~\ref{lem:cycle} to prove that the limiting
approximation solution is not locally of bounded variation.
\begin{proof}[Proof of Theorem \ref{thm:limit}]
We use the same notations as in the proof of Theorem \ref{thm:smooth}.
By Lemma \ref{lem:cycle}(iii), all child cubes of a parent cube lie in the central $x_1$-slab of relative width $\gamma=3/10$. 
Hence,
\begin{equation}\label{eq:volumeupper}
 |\Omega_{m+1}|\leq\frac3{10}|\Omega_m|
 \leq\left(\frac3{10}\right)^{m+1}.
\end{equation}
For $n\geq m$, all changes after stage $m$ are supported in $\Omega_m$. 
Thus,
\[
 D^2u_n=D^2u_m,
 \qquad
 A_n=A_m
 \qquad\text{on }B_2\setminus\Omega_m.
\]
Since $I\leq A_m\leq2^{81}I$,
\[
 \|A_n-A_m\|_{L^1(B_2)}
 \leq2\sqrt3\,2^{81}|\Omega_m|
 \leq\sqrt3\,2^{82}\left(\frac3{10}\right)^m.
\]
Therefore, $A_m\to A$ in $L^1(B_2)$ for some measurable symmetric matrix $A$, and after passing to a subsequence,
\[
 I\leq A\leq2^{81}I
\]
almost everywhere. Since $u_m\to u$ uniformly, $u$ is an $L^1$-approximation solution of \[A(x):D^2u=0 \qquad\text{in }B_2.\]

It remains to show that $u\notin\BV_{\rm loc}(B_1)$. 
For every $Q\in\Q_m$, let $G_Q$ be the untouched cube from Lemma \ref{lem:cycle}. 
All descendants of $Q$ stay in the central slab and therefore never meet $G_Q$. 
Hence, $u$ is smooth near $G_Q$ and
\[
 Du(x_Q'+y)=s_mM_0y+b_Q'
\]
there. 
Here, $x_Q'$ is the center of $G_Q$ and $b_Q'\in \R^3$.
The cubes $G_Q$, over all generations, are pairwise disjoint.
Since $G_Q$ has side $\sigma r_m$, the same argument as in \eqref{eq:ycube} gives
\[
 \int_{G_Q}|Du|\,dx
 \geq\frac{\sigma^4}{16}c^m r_m|Q|.
\]
Summing over $Q\in\Q_m$ yields
\begin{equation}\label{eq:nonBV}
 \sum_{Q\in\Q_m}\int_{G_Q}|Du|\,dx
 \geq\frac{\sigma^4}{16}(\phi\kappa\sqrt c)^m
 \longrightarrow\infty.
\end{equation}
All the $G_Q$ lie in the fixed cube $Q_0\Subset B_1$. 
If $u$ belonged to $\BV$ on a neighborhood of $Q_0$, its total variation would dominate the left-hand side of \eqref{eq:nonBV} for every $m$, a contradiction. 
This proves Theorem \ref{thm:limit}.
\end{proof}

Below, we give a series of remarks concerning Theorems \ref{thm:smooth}--\ref{thm:limit}.

\begin{remark}[Failure of weak-$L^1$ estimates]\label{rem:weakL1}
Our construction in fact rules out even a coefficient-independent weak-$L^1$ estimate for the gradient.
Indeed, on every active cube $Q\in\mathcal Q_m$, centered at $x_Q$, we have
\[
Du_m(x_Q+y)=s_mM_0y+b_Q,
\qquad |s_m|=\frac12 c^m.
\]
By symmetry,
\[
|s_mM_0y+b_Q|+|-s_mM_0y+b_Q|
\ge 2|s_m||y|.
\]
Hence, on a fixed positive fraction of $Q$,
\[
|Du_m|\ge C |s_m|r_m.
\]
Since the total volume of the active cubes is at least $\phi^m|Q_0|$, it follows that
\[
\|Du_m\|_{L^{1,\infty}(B_1)}
\ge C |s_m|r_m\phi^m
= C(\phi\kappa\sqrt c)^m\longrightarrow\infty.
\]
Thus, there is no estimate of the form
\[
\|Du\|_{L^{1,\infty}(B_1)}
\le C(\Lambda)\|u\|_{L^\infty(B_2)}.
\]
In particular, ellipticity alone gives no estimate in any Lorentz space stronger than weak-$L^1$.

\medskip

Our construction disproves estimates whose constants depend only on dimension and ellipticity; it does not address estimates with additional dependence on a modulus of continuity or derivative bounds of the coefficients.
\end{remark}

\begin{remark}[The measurable limit]
Let
\[
\Omega_m=\bigcup_{Q\in\mathcal Q_m}Q,
\qquad
E=\bigcap_{m=0}^\infty \overline \Omega_m.
\]
Since $|\Omega_m|\to0$, we have $|E|=0$. 
Moreover, all changes after step $m$ are supported in $\Omega_m$. 
Hence, for every $x\in B_2\setminus E$, the sequences $(u_m)$ and $(A_m)$ stabilize in a neighborhood of $x$. Consequently, $u,A\in C^\infty(B_2\setminus E)$ and
\begin{align*}
A:D^2u=0
\qquad\text{classically in }B_2\setminus E.
\end{align*}
Thus, the limiting solution is smooth away from a set of measure zero, although
\[
u\notin BV_{\rm loc}(B_1).
\]

The convergence of the coefficients is also stronger than the $L^1$ convergence stated above. 
Since $A_n=A_m$ on $B_2\setminus\Omega_m$ for every $n\ge m$, while the coefficients are uniformly bounded, for every $1\le p<\infty$,
\[
\|A_n-A_m\|_{L^p(B_2)}
\le C|\Omega_m|^{1/p}\longrightarrow0.
\]
Therefore,
\[
A_m\longrightarrow A
\qquad\text{ in $L^p(B_2)$  for every }1\le p<\infty.
\]
Hence, even strong $L^p$ convergence of smooth uniformly elliptic coefficients for every finite $p$, together with uniform convergence of the corresponding solutions and common boundary data, does not yield local $BV$ compactness of the solutions.
\end{remark}

\begin{remark}
The explicit constant $2^{81}$ is not intended to be sharp. 
The exact value of the smallest possible ellipticity ratio for which an ellipticity-only $W^{1,1}$ estimate can fail remains a separate question.
\end{remark}

\begin{remark}[Optimal exponent in $W^{1,\delta}$ estimates] Given the failure of interior $W^{1,1}$ estimates for uniformly elliptic equations in nondivergence form with large ellipticity ratio $\Lambda$ in dimension three, a natural question is now to find the best exponent $\delta(\Lambda)\in (0, 1)$ for which interior $W^{1,\delta}$ estimates hold for these equations. 
This question about solutions with low-integrability derivatives is related to Lin's $L^p$ estimates for second derivatives for some $p>0$ depending on ellipticity; see \cite{Lin86}. One can also ask for the best exponent $p$ here.
\end{remark}

\end{document}